\documentclass[reqno,11pt]{amsart}
\usepackage[foot]{amsaddr}
\usepackage{bbold}
\usepackage{charter}
\usepackage[margin=1in]{geometry}
\usepackage[colorlinks=true,linktoc=all,linkcolor=blue,citecolor=blue]{hyperref}

\theoremstyle{plain}
\newtheorem{theorem}{Theorem}[section]
\newtheorem*{theorem*}{Theorem}
\newtheorem{prop}[theorem]{Proposition}
\newtheorem{lemma}[theorem]{Lemma}
\newtheorem{cor}[theorem]{Corollary}
\theoremstyle{definition}

\newtheorem*{definition*}{Definition}

\numberwithin{equation}{section}
\everymath{\displaystyle}
\allowdisplaybreaks

\newcommand{\D}{\mathbb D}

\newcommand{\A}{\alpha}

\newcommand{\DI}{\mathcal{D}}
\newcommand{\DA}{\DI_{\A}}
\newcommand{\Ph}{\varphi}
\newcommand{\CO}{C_{\Ph}}
\newcommand{\CP}{C_{\psi}}
\newcommand{\XCN}{X_0\cap X_1}
\newcommand{\XSN}{X_0+X_1}
\newcommand{\BY}{\mathcal{B}(\mathcal{Y})}

\newcommand{\NTO}{\|T_1\|_{\mathcal{B}(X_1)}}
\newcommand{\NTZ}{\|T_0\|_{\mathcal{B}(X_0)}}
\newcommand{\NTT}{\|T_t\|_{\mathcal{B}(X_t)}}

\title[Compact Differences of Composition Operator]{Compact Differences of Composition Operators\\ on Weighted Dirichlet Spaces}

\author{Robert F.~Allen\textsuperscript{1}, Katherine C.~Heller\textsuperscript{2}, and Matthew A.~Pons\textsuperscript{2}}
\address{\textsuperscript{1}Department of Mathematics and Statistics, University of Wisconsin-La Crosse}
\address{\textsuperscript{2}Department of Mathematics, North Central College}

\email{rallen@@uwlax.edu, kheller@noctrl.edu, mapson@noctrl.edu}


\begin{document}

\begin{abstract}
Here we consider when the difference of two composition operators is compact on the weighted Dirichlet spaces $\DA$.  Specifically we study differences of composition operators on the Dirichlet space $\mathcal{D}$ and $S^2$, the space of analytic functions whose first derivative is in $H^2$, and then use Calder\'{o}n's complex interpolation to extend the results to the general weighted Dirichlet spaces. As a corollary we consider composition operators induced by linear fractional self-maps of the disk.
\end{abstract}

\maketitle

\section{Introduction}

For an analytic self-map $\Ph$ of the unit disk $\D$ and a Banach space $\mathcal{Y}$ of functions analytic on the unit disk, we define the composition operator $\CO$ with symbol $\Ph$ by the rule $\CO f=f\circ \Ph$ for all $f\in\mathcal{Y}$.  The study of these types of operators began formally with Nordgren's paper \cite{EN} where he explored properties of composition operators acting on the Hardy Hilbert space $H^2$. Over the past fifty years, the study has proved to be a lively source of inquiry, most likely due to the fact that the study of such operators lies at the intersection of complex function theory and operator theory.  With this perspective, the goal of such an investigation seeks to relate the operator properties of $\CO$ to the analytic and geometric properties of the symbol function $\Ph$.  In this note we will focus on the property of compactness.

Recall that an operator $T$ acting on a Banach space $\mathcal{Y}$ is compact if it takes the unit ball in $\mathcal{Y}$ (which is not compact in the infinite dimensional setting) into a set with compact closure. For composition operators, compactness is generally classified by how the symbol function behaves near the boundary of the disk.  For instance, it is well known that a composition operator $\CO$ is compact on $H^{\infty}(\D)$ if and only if $\|\Ph\|_{\infty}<1$.  This condition is sufficient on many other spaces, but is often not necessary, meaning that the symbol can have some contact with the boundary and still induce a compact composition operator.  This phenomenon has been studied in depth on the Hardy and weighted Bergman spaces;  in \cite{MCSH} the authors supply the intuitive message for the case of the Hardy spaces: ``$\CO$ will be compact on $H^p$ if and only if $\Ph$ squeezes the unit disc rather sharply into itself".  They then make sense of this intuitive notion using the finite angular derivative of the symbol $\Ph$.

Shapiro and Taylor were the first to observe the role that the angular derivative plays in the study of the compactness problem.  In \cite{JSPT} they showed that the symbol of a compact composition operator $\CO$ on $H^2$ cannot have finite angular derivative at any point of $\partial\D$.  This result was extended to the Bergman space $A^2$ by Boyd in \cite{DB}.  In \cite{MCSH}, the authors use Carleson measure techniques to show that nonexistence of the angular derivative of $\Ph$ is also a sufficient condition for $\CO$ to be compact on $A^2$, however this is not the case on $H^2$ and the authors provide an example demonstrating such a $\Ph$.  Shapiro later characterized the compact composition operators on $H^2$ in terms of the Nevanlinna counting function and the essential norm of the operator in \cite{JS}.

In contrast with this, for $\CO$ acting on the space $S^2$, the situation is much simpler.  First, if $\CO$ is bounded on $S^2$, then $\Ph$ must have finite angular derivative at any point in $\partial\D$ which is mapped to $\partial\D$; see \cite{CMC} Theorem 4.13.  For compactness, it turns out that $\CO$ is compact on $S^2$ if and only if $\|\Ph\|_{\infty}<1$; see, for example, \cite{JS1}. Thus we see a drastic (and interesting) shift in behavior among spaces that are closely related to each other.

Here we are interested in determining when the difference of two composition operators is compact. In \cite{BM} MacCluer investigated this on the Hardy space to understand the topological structure of the collection of compact composition operators within the set of all composition operators. More recently, Moorhouse considered this on a broader range of spaces in \cite{JMoore} and further considered the role of the second order data of the symbol $\Ph$; Bourdon also considered this same question on the Hardy space in \cite{PB}.  Here we aim to extend some of those results.

In the next section we gather the necessary prerequisites.  In Section 3 we work primarily on the Dirichlet space and $S^2$.  Our techniques mimic those of MacCluer and Moorhouse, but required a change in perspective due to the behavior of the reproducing kernels in our spaces of interest.  To overcome this, we instead use the kernels for evaluation of the first derivative.  Finally, we appeal to Calder\'{o}n's method of complex interpolation to provide an extension to the general weighted Dirichlet spaces.  This work is, in part, an invitation for other researchers to employ these newer techniques to the study of composition operators.

\section{Preliminaries}

\subsection{Spaces of analytic functions}

We let $\D$ denote the open unit disk in the complex plane, $\D=\left\{z\in\mathbb{C}:|z|<1\right\}$, and let $H(\D)$ be the space of functions analytic on $\D$. The following classical spaces of analytic functions have received much attention in the study of composition operators. The Hardy space is defined by $$H^2(\D)=\left\{f \textup{ in }H(\D):\|f\|_{H^2}^2=\lim_{r\rightarrow1^-}\int_0^{2\pi}|f(re^{i\theta})|^2\,\frac{d\theta}{2\pi}<\infty\right\}$$ where $d\theta$ is the Lebesgue arc-length measure on the unit circle. For $\beta>-1$, the standard weighted Bergman space is defined by $$A_{\beta}^2(\D)=\left\{f \textup{ in }H(\D) :\|f\|_{A_{\beta}^2}^2=\int_{\D}|f(z)|^2(1-|z|^2)^{\beta}\,dA<\infty\right\}$$ where $dA$ is the Lebesgue area measure normalized so that $A(\D)=1$; the Dirichlet space is given by $$\DI(\D)=\left\{f \textup{ in }H(\D) :\|f\|_{\DI}^2=|f(0)|^2+\int_{\D}|f'(z)|^2\,dA<\infty\right\}.$$

Recall that a reproducing kernel Hilbert space $\mathcal{H}$ with inner product $\langle\cdot,\cdot\rangle_{\mathcal{H}}$ has the property that for each $w\in\D$, there is a unique function $K_w\in \mathcal{H}$ such that $$f(w)=\langle f, K_w\rangle_{\mathcal{H}}.$$  For the Hardy and weighted Bergman spaces, the kernels have a similar form: $$K_w(z)=\frac{1}{1-\overline{w}z}$$ on $H^2$ and $$K_w(z)=\frac{1}{(1-\overline{w}z)^{\beta+2}}$$ on $A_{\beta}^2$ with $\beta>-1$.  On the Dirichlet space the kernel takes on a more complicated form, $$K_w(z)=1+\log\frac{1}{1-\overline{w}z},$$ where $\log z$ denotes the principal branch of the logarithm.

Though often convenient from the computational point of view, presenting the norms for these spaces in terms of integrals obscures the relationship between the spaces, though it is somewhat revealed in the representations of the reproducing kernels. To make the relationship more explicit, we can consider the spaces with a series norm, (equal to the norm given above for the Hardy and Dirichlet spaces, but only equivalent to the Bergman norm): $$H^2(\D)=\left\{f(z)=\sum_{n=0}^{\infty}a_nz^n \textup{ in }H(\D) :\sum_{n=0}^{\infty}|a_n|^2<\infty\right\};$$
$$A_{\beta}^2(\D)=\left\{f(z)=\sum_{n=0}^{\infty}a_nz^n \textup{ in }H(\D) :|a_0|^2+\sum_{n=1}^{\infty}\frac{|a_n|^2}{n^{\beta+1}}<\infty\right\};$$
$$\DI(\D)=\left\{f(z)=\sum_{n=0}^{\infty}a_nz^n \textup{ in }H(\D) :|a_0|^2+\sum_{n=1}^{\infty}n|a_n|^2<\infty\right\}.$$  With these characterizations, we see the obvious containment relationship $\DI\subset H^2\subset A^2,$ but more importantly it is apparent that there are other spaces that deserve consideration. One particular space that has received more attention as of late is $S^2$ which can be defined with a series norm or an equal integral norm, $$\begin{aligned}S^2(\D)&=\left\{f \textup{ in }H(\D) :\|f\|_{S^2}^2=|f(0)|^2+\|f'\|_{H^2}^2<\infty\right\}\\
&=\left\{f(z)=\sum_{n=0}^{\infty}a_nz^n \textup{ in }H(\D) :\|f\|_{S^2}^2=|a_0|^2+\sum_{n=1}^{\infty}n^2|a_n|^2<\infty\right\}.\end{aligned}$$ While this space is also a reproducing kernel Hilbert space, one of the first difficulties encountered in this setting is that there is not a ``nice'' closed form for the reproducing kernel functions with respect to this norm. The reason for this is the fact that, on $S^2$, the kernel for evaluation at $w$ takes the form $$K_w(z)=1+\sum_{n=1}^{\infty}\frac{(\overline{w}z)^n}{n^2},$$ however we cannot identify this sum as an elementary function; for more on this, see \cite{KH}.  We will discuss how to overcome this obstacle shortly.

In general, for $\A\geq-1$ we define the weighted Dirichlet space $$\DA(\D)=\left\{f(z)=\sum_{n=0}^{\infty}a_nz^n \textup{ in }H(\D) :\|f\|_{\DA}^2=|a_0|^2+\sum_{n=1}^{\infty}n^{1-\A}|a_n|^2<\infty\right\}.$$  These are all Hilbert spaces and we see that $H^2=\DI_1$ with equal norm; for $\beta>-1$, the standard weighted Bergman space $A_{\beta}^2=\DI_{\beta+2}$ with an equivalent norm. Also, $\DI=\DI_0$ and $S^2=D_{-1}$ with equal norm. Moreover, if $-1\leq\A<\beta<\infty$, $\DA\subset\DI_{\beta}$ with continuous inclusion and the analytic polynomials are dense in $\DA$.

As is the case with $S^2=D_{-1}$, there are no nice closed forms for the reproducing kernels for $\DA$ when $-1<\A<0$ which is a drawback since these kernels are quite useful in the study of composition operators.  In particular, if $\mathcal{H}$ is a functional Hilbert space of functions on the disk and $\CO$ is bounded on $\mathcal{H}$, then for $w\in \D$, we have $$\CO^*K_w=K_{\varphi(w)}.$$ To overcome this drawback, in many instances we will consider the linear functional for evaluation of the first derivative at a point in the disk; for a reference see \cite{CMC} Theorem 2.16. If $\mathcal{H}=\DA$ for $\A\geq -1$, these functionals are bounded, and thus the Riesz Representation Theorem (\cite{Conway:func} Theorem I.3.4) guarantees the existence of a function, denoted $K_w^{(1)}$ for $w\in\D$, such that $$f'(w)=\langle f,K_w^{(1)}\rangle_{\mathcal{H}}.$$ As with the point evaluation kernels, these kernels behave predictably under the action of the adjoint of a composition operator and it is easy to see that $$\CO^*K_w^{(1)}=\overline{\Ph'(w)}K_{\Ph(w)}^{(1)}.$$  In particular, we will employ these in the spaces $\DI$ and $S^2$.  On $\DI$, we find that $$K_w^{(1)}(z)=\frac{z}{1-\overline{w}z} \hspace{.2in}\textup{and}\hspace{.2in}\|K_w^{(1)}\|_{\DI}^2=\frac{1}{(1-|w|^2)^2},$$ whereas on $S^2$ we have $$K_w^{(1)}(z)=\frac{1}{\overline{w}}\log\frac{1}{1-\overline{w}z} \hspace{.2in}\textup{and}\hspace{.2in}\|K_w^{(1)}\|_{S^2}^2=\frac{1}{1-|w|^2}.$$

\subsection{Julia Carath\'{e}odory Theorem}
For an analytic self-map $\Ph$ of the unit disk, the angular derivative plays a key role in determining compactness of composition operators on many of the spaces in question.  For $\zeta\in\partial\D$ and $M>1$, a nontangential approach region at $\zeta$ is defined by $$\Gamma(\zeta, M)=\{z\in\D:|z-\zeta|<M(1-|z|)\}$$ and a function $f$ has a nontangential limit at $\zeta$ if $\lim_{z\rightarrow \zeta}f(z)$ exists in each nontangential region $\Gamma(\zeta, M)$.  When a nontangential limit exists, we denote it by $\angle\lim_{z\rightarrow\zeta}f(z).$  Furthermore, if $\Ph$ is a self-map of the disk and $\zeta\in\partial\D$, then $\Ph$ has finite angular derivative at $\zeta$ if there exists $\eta\in\partial\D$ such that $$\Ph'(\zeta):=\angle\lim_{z\rightarrow\zeta}\frac{\eta-\Ph(z)}{\zeta-z}$$ exists as a (finite) complex value. One obvious implication of the existence of a finite angular derivative for $\Ph$ at $\zeta$ is that $\Ph$ has nontangential limit of modulus 1 at $\zeta.$  The Julia-Carath\'{e}odory Theorem provides several other implications; for a reference see \cite{CMC} Theorem 2.44.

\begin{theorem}[Julia-Carath\'{e}odory Theorem]\label{thm:JC}
For an analytic self-map $\Ph$ of $\D$ and $\zeta\in\partial\D$, the following are equivalent:
\begin{enumerate}
\item[(a)] $d(\zeta):=\liminf_{z\rightarrow\zeta}(1-|\Ph(z)|)/(1-|z|)<\infty$, where the limit is taken as $z\rightarrow\zeta$ unrestrictedly in $\D$;

\item[(b)] $\Ph$ has finite angular derivative $\Ph'(\zeta)$ at $\zeta$;

\item[(c)] both $\Ph$ and $\Ph'$ have finite nontangential limits at $\zeta$, with $|\eta|=1$ where $\eta=\lim_{r\rightarrow 1}\Ph(r\zeta).$
\end{enumerate}
Moreover, when these conditions hold, we have $\angle\lim_{z\rightarrow\zeta}\Ph'(z)=\Ph'(\zeta)=\overline{\zeta}\eta d(\zeta)$, i.e. $d(\zeta)=|\Ph'(\zeta)|>0$, and $d(\zeta)=\angle\lim_{z\rightarrow\zeta}(1-|\Ph(z)|)/(1-|z|).$
\end{theorem}

In characterizing compact differences, we will be interested in maps which have similar behavior on the boundary of $\D$.  If $\Ph$ and $\psi$ are two self-maps of the unit disk both with finite angular derivative at $\zeta\in\partial\D$, then we say that the maps have the same \textit{first order data} at $\zeta$ if $\Ph(\zeta)=\psi(\zeta)$ (as radial limits) and $\Ph'(\zeta)=\psi'(\zeta)$.  If in addition $\Ph$ and $\psi$ are twice differentiable at $\zeta$ (meaning that if we consider $\Ph$ and $\psi$ as functions on $\D\cup\{\zeta\}$, then they are twice continuously differentiable) with $\Ph''(\zeta)=\psi''(\zeta)$, then we say that $\Ph$ and $\psi$ have the same \textit{second order data} at $\zeta$.

\subsection{Linear fractional self-maps of the disk}

Recall that a linear fractional map has the form $\Ph(z)=(az+b)/(cz+d)$; the condition that $ad-bc\neq 0$ is necessary and sufficient for such a $\Ph$ to be a univalent, nonconstant mapping of the Riemann sphere onto itself. Our focus here is on linear fractional self-maps of the disk and we point the reader to  Chapter 0 of \cite{JSComp} for more information.  With this narrowed focus we may assume that $d\neq 0$, in which case we can represent $\Ph$ in the form  $\Ph(z)=(az+b)/(cz+1)$.

It is easy to see that every linear fractional self-map $\Ph$ of $\D$ will induce a bounded composition operator on the weighted Dirichlet spaces under consideration. This is due to the fact that the map $\Ph'$ is continuous, and hence bounded, on $\overline{\D}$.  Of particular interest here is the role of second order data. The following statement seems to be known but we were unable to find a proof in the literature.

\begin{lemma}\label{lem:linearfractional1}
If $\Ph$ and $\psi$ are linear fractional self-maps of $\D$ with the same second order data at a point in $\partial\D$, then $\Ph=\psi$.
\end{lemma}

\begin{proof}
Assume that $\Ph(z)=(az+b)/(cz+1)$ and $\psi(z)=(\A z+\beta)/(\gamma z+1)$ are nonconstant linear fractional self-maps of $\D$ with the same second order data.  By composing with rotations, we may assume without loss of generality that $\Ph(1)=\psi(1)=1$.  Thus we have the following assumptions:
\begin{enumerate}
\item[(i)] $\displaystyle \Ph(1)=\frac{a+b}{c+1} = \frac{\A+\beta}{\gamma+1}=\psi(1)=1$;\vspace{.1in}
\item[(ii)] $\displaystyle \Ph'(1)= \frac{a-bc}{(c+1)^2} = \frac{\A-\beta\gamma}{(\gamma+1)^2} =\psi'(1)$;\vspace{.1in}
\item[(iii)] $\displaystyle \Ph''(1)= \frac{-2c(a-bc)}{(c+1)^3} = \frac{-2\gamma(\A-\beta\gamma)}{(\gamma+1)^3} =\psi''(1)$.
\end{enumerate}

\noindent First note that $c,\gamma\neq -1$ since each of the above quantities exist as finite complex values.  Now, considering (ii) and (iii), it immediately follows that $$\frac{c}{c+1}=\frac{\gamma}{\gamma+1}$$ and hence $c=\gamma$.  Substituting this into (i), we have $a+b=\A+\beta$ or $a-\A=\beta-b$.  Moreover, this same substitution in (ii) implies that $$a-\A=bc-\beta c =c(b-\beta)$$ or $$\beta- b=c(b-\beta).$$  Thus it must be the case that $b=\beta$, since $c\neq -1$.  It is then immediate that $a=\A$ and hence $\Ph=\psi$.
\end{proof}

\subsection{Calder\'{o}n's complex interpolation}

Let $(X_0, \|\cdot\|_0)$ and $(X_1, \|\cdot\|_1)$ be a compatible pair of Banach spaces in the sense of Calder\'{o}n (see \cite{AC}). Both $X_0$ and $X_1$ may be continuously embedded in the complex topological vector space $\XSN$ when equipped with the norm $$\|x\|_{\XSN}=\inf\left\{\|y\|_0+\|z\|_1:x=y+z, y\in X_0, z\in X_1\right\}.$$ In addition, the space $\XCN$, with norm $$\|x\|_{\XCN}=\max\left\{\|x\|_0,\,\|x\|_1\right\},$$ maps continuously into $X_0$ and $X_1$.  In this note we will further assume that the space $\XCN$ is dense in both $X_0$ and $X_1$ and define the interpolation algebra $\mathcal{I}[X_0,X_1]$ to be the set of
all linear operators $T:\XCN\rightarrow\XCN$ that are both 0-continuous and 1-continuous. The interpolation algebra defined above first appeared in the $L^p$-space setting in \cite{BB}; for properties and applications to the study of spectra, see \cite{BB1}, \cite{HS}, \cite{MP}, and \cite{KS}.

For a Banach space $\mathcal{Y}$, we let $\BY$ denote the set of all bounded operators on $\mathcal{Y}$. By continuity any operator
$T\in\mathcal{I}[X_0,X_1]$ induces a unique operator $T_i\in\mathcal{B}(X_i)$, $i=0,1$. For $t\in(0,1)$, let $X_t=[X_0,X_1]_t$ be the interpolation space obtained via Calder\'{o}n's method of complex interpolation; it follows then that $\XCN$ is dense in $X_t$ and $T$ also induces a unique operator $T_t\in\mathcal{B}(X_t)$ satisfying $$\NTT\leq \NTZ^{1-t}\NTO^t, \hspace{.1in} t\in(0,1).$$

To apply interpolation techniques to our study, we first verify that the weighted Dirichlet spaces can be interpreted as interpolation spaces. One can see this by considering these spaces as weighted $\ell^2$-spaces or by considering the techniques developed in \cite{JMC}; a direct proof of this nature can  be found in \cite{MP}.

\begin{prop}\label{propDAinterp} Suppose $-1<\A<\gamma<\beta<\infty$.  If $t\in(0,1)$ with $\gamma=(1-t)\A+t\beta$, then $[\DA,\DI_{\beta}]_t=\DI_{\gamma}$ with the series norm given above.
\end{prop}

One appealing property of working with composition operators on these spaces in the interpolation setting is the nested behavior mentioned earlier.  Specifically, $\DA\subset \DI_{\beta}$, for $-1<\A<\beta$, with continuous inclusion, so $\DA=\DA\cap \DI_{\beta}$.  Combining this with the fact that the analytic polynomials are dense in each weighted Dirichlet space implies that $\DA\cap \DI_{\beta}$ is dense in $\DA$ and $\DI_{\beta}$.  Moreover, MacCluer and Shapiro showed in \cite{MCSH} that boundedness of $\CO$ on $\DA$ implies boundedness on $\DI_{\beta}$ when $-1<\A<\beta.$ Thus, for our purposes, it suffices to know that $\CO$ is bounded on the single endpoint space $\DA$. Furthermore, the fact that we are defining our operators on a dense subset of each space implies that for $t\in(0,1)$ the interpolated operator satisfies $(\CO)_t =\CO$.

When using interpolated operators, the goal is to determine properties of the operator on an interpolation space $X_t$ based on properties of the operator on the endpoint spaces, or to extrapolate properties from one interpolation space to the other interpolation spaces and/or the endpoint spaces.  The result that we will make use of here is Cwikel's compactness result which extrapolates compactness on one interpolation space to the other interpolation spaces, but not necessarily to the endpoint spaces.

\begin{theorem}[\cite{MC1} Theorem 2.1]\label{thm:Cwikelcompact}
If $T_0$ and $T_1$ are bounded and $T_t$ is compact for some $t\in(0,1)$, then $T_x$ is compact for all $x\in(0,1)$.
\end{theorem}

\section{Compact Differences of Composition Operators}

As mentioned previously, our goal is to extend the work of Moorhouse \cite{JMoore}.  There the author characterized when two composition operators have compact difference on the weighted Bergman spaces ($\A>1$ in our scale of weighted Dirichlet spaces) and provided partial results for weighted Dirichlet spaces in the range $0<\A\leq 1$.  Here we extend some of those results to the entire range of weighted Dirichlet spaces.  We first work on the Dirichlet space by modifying techniques used in the Hardy and weighted Bergman space setting.  We then discuss extending this to the space $S^2$. Finally, we discuss the compact difference problem on an arbitrary weighted Dirichlet space where we apply Calder\'{o}n's complex interpolation.

\begin{lemma}\label{lem:Dfod}
Let $\Ph$ and $\psi$ be analytic self-maps of $\D$ such that $\CO$ and $\CP$ are bounded on $\DI$. Further assume that $\Ph$ and $\psi$ have finite angular derivative at some point $\zeta\in\partial\D$.  If $\CO-\CP$ is compact on $\DI$, then $\Ph$ and $\psi$ have the same first order data at $\zeta$.
\end{lemma}

\begin{proof}
First note that it suffices to prove the statement in the case when $\zeta=1$ since the rotations of the disk give rise to unitary operators on $\DI$.  By the Julia-Carath\'{e}odory Theorem it follows that $|\Ph(1)|=|\psi(1)|=1$ and by hypothesis $\Ph'(1)=s<\infty$ and $\psi'(1)=t<\infty$. For the sake of notation we set $\Ph(1)=\eta$.  Now, we will show that if $\Ph$ and $\psi$ do not have the same first order data then $\CO-\CP$ is not compact by showing that the essential norm of $\CO-\CP$ is bounded away from 0; in particular we will show that $$\|\CO-\CP\|_e^2\geq1.$$  To this end, we first obtain lower estimates on the norm of $(\CO-\CP)^*$ acting on the normalized reproducing kernels for evaluation of the first derivative; for $w\in\D$ we have $$\left\|(\CO-\CP)^*\left(\frac{K_w^{(1)}}{\|K_w^{(1)}\|}\right)\right\|^2=\frac{\|\CO^*K_w^{(1)}-\CP^*K_w^{(1)}\|^2}{\|K_w^{(1)}\|^2}$$ and
$$\begin{aligned}\frac{\|\CO^*K_w^{(1)}-\CP^*K_w^{(1)}\|^2}{\|K_w^{(1)}\|^2}=\frac{\|\CO^*K_w^{(1)}\|^2}{\|K_w^{(1)}\|^2}+
\frac{\|\CP^*K_w^{(1)}\|^2}{\|K_w^{(1)}\|^2}-2\textup{Re}\left\langle\frac{\CO^*K_w^{(1)}}{\|K_w^{(1)}\|},\frac{\CP^*K_w^{(1)}}{\|K_w^{(1)}\|}\right\rangle.
\end{aligned}$$ Considering the action of the adjoint and the formulas for these kernels and their norms in $\DI$, we see that previous line is equal to
$$\frac{|\Ph'(w)|^2(1-|w|^2)^2}{(1-|\Ph(w)|^2)^2}+\frac{|\psi'(w)|^2(1-|w|^2)^2}{(1-|\psi(w)|^2)^2}-
2\textup{Re}\left(\frac{\overline{\Ph'(w)}\psi'(w)(1-|w|^2)^2}{(1-\overline{\Ph(w)}\psi(w))^2}\right).$$  By the Julia-Carath\'{e}odory Theorem $$\angle\lim_{w\rightarrow 1}\frac{1-|w|^2}{1-|\Ph(w)|^2}=\frac1s\hspace{.2in}\textup{and}\hspace{.2in}\angle\lim_{w\rightarrow 1}\Ph'(w)=\Ph'(1)=s,$$ from which it follows that $$\angle\lim_{w\rightarrow1}\frac{|\Ph'(w)|^2(1-|w|^2)^2}{(1-|\Ph(w)|^2)^2}=1.$$

Next we consider two cases.  If $\psi(1)\neq \eta$  (as a radial limit), then we can find a sequence $\{r_n\}$ increasing to 1 such that $\lim_{n\rightarrow\infty}\psi(r_n)\neq \eta$. Then $\lim_{n\rightarrow\infty}(1-\overline{\Ph(r_n)}\psi(r_n))\neq 0$ and the fact that $s,t<\infty$ guarantee us that $$\lim_{n\rightarrow\infty}2\textup{Re}\left(\frac{\overline{\Ph'(r_n)}\psi'(r_n)(1-|r_n|^2)^2}{(1-\overline{\Ph(r_n)}\psi(r_n))^2}\right)=0.$$  This shows that in this case \begin{equation}\label{eqn:adjointkernel}\left\|(\CO-\CP)^*\left(\frac{K_w^{(1)}}{\|K_w^{(1)}\|}\right)\right\|\geq 1.\end{equation}

For the second case suppose that $\psi(1)=\eta$ (as a radial limit) but $s\neq t$. Observe that $$\frac{(1-\overline{\Ph(w)}\psi(w))^2}{\overline{\Ph'(w)}\psi'(w)(1-|w|^2)^2}=\frac1{\overline{\Ph'(w)}\psi'(w)}\left(\frac{1-\overline{\Ph(w)}\psi(w)}
{1-|w|^2}\right)^2$$ which equals $$\frac1{\overline{\Ph'(w)}\psi'(w)}\left(\frac{1-|\Ph(w)|^2}{1-|w|^2}+
\frac{\overline{\Ph(w)}(1-w)}{1-|w|^2}\left(\frac{\eta-\psi(w)}{1-w}-\frac{\eta-\Ph(w)}{1-w}\right)\right)^2.$$  For $M>1$, consider the boundary of a nontangential approach region $$\gamma_M=\left\{w\in\D:\frac{|1-w|}{1-|w|^2}=M\right\}.$$ As $w\rightarrow 1$ along $\gamma_M$, the Julia-Carath\'{e}odory Theorem guarantees us that $$\left|\frac{\overline{\Ph(w)}(1-w)}{1-|w|^2}\left(\frac{\eta-\psi(w)}{1-w}-\frac{\eta-\Ph(w)}{1-w}\right)\right|\rightarrow M|t-s|.$$ Thus for $N>0$, by choosing $M$ sufficiently large, we may find a sequence $\{w_n\}$ approaching 1 along $\gamma_M$ such that for $n$ large enough it follows that $$\left|\frac{(1-\overline{\Ph(w_n)}\psi(w_n))^2}{\overline{\Ph'(w_n)}\psi'(w_n)(1-|w_n|^2)^2}\right|>N.$$  Equivalently, for $0<\varepsilon<1$, we may find a sequence $\{w_n\}$ converging to 1 nontangentially such that $$2\textup{Re}\left(\frac{\overline{\Ph'(w_n)}\psi'(w_n)(1-|w_n|^2)^2}{(1-\overline{\Ph(w_n)}\psi(w_n))^2}\right)\leq\left|\frac{2\overline{\Ph'(w_n)}\psi'(w_n)(1-|w_n|^2)^2}
{(1-\overline{\Ph(w_n)}\psi(w_n))^2}\right|< \varepsilon$$ for $n$ sufficiently large.  Thus we see that $$\left\|(\CO-\CP)^*\left(\frac{K_w^{(1)}}{\|K_w^{(1)}\|}\right)\right\|\geq 1-\varepsilon$$ for all $\varepsilon$ with $0<\varepsilon<1$, and hence in this case the estimate from ~\eqref{eqn:adjointkernel} also holds.

To show that $\|\CO-\CP\|_e^2\geq 1$, recall that the normalized kernel functions $K_w^{(1)}/\|K_w^{(1)}\|$ converge weakly to zero as $|w|\rightarrow 1$ (see \cite{CMC} Proposition 7.13).  If we then consider any compact operator $Q$, it must then be the case that $$\left\|Q^*\left(\frac{K_w^{(1)}}{\|K_w^{(1)}\|}\right)\right\|\rightarrow 0$$ as $|w|\rightarrow 1$.  Combining this with the estimate $$\|\CO-\CP-Q\|\geq\left\|(\CO-\CP)^*\left(\frac{K_w^{(1)}}{\|K_w^{(1)}\|}\right)\right\|-\left\|Q^*\left(\frac{K_w^{(1)}}{\|K_w^{(1)}\|}\right)\right\|$$ it follows that $$\|\CO-\CP\|_e^2=\inf\{\|\CO-\CP-Q\|^2:Q \textup{ compact}\}\geq 1$$ completing the proof.\end{proof}

Our next result is a Dirichlet space analog of \cite{JMoore} Theorem 4. The statement there concerns weighted Dirichlet spaces $\DA$ with $\alpha>0$. Our statement takes a slightly different form in that the derivatives of $\Ph$ and $\psi$ appear; this is due to the fact that we are using the kernels for evaluation of the first derivative.

\begin{theorem}\label{thm:Dpseudo}
Let $\Ph$ and $\psi$ be analytic self-maps of $\D$ such that $\CO$ and $\CP$ are bounded on $\DI$ and define $$\rho(z)=\left|\frac{\Ph(z)-\psi(z)}{1-\overline{\Ph(z)}\psi(z)}\right|$$ for $z\in\D$.  If $\CO-\CP$ is compact on $\DI$, then $$\lim_{|z|\rightarrow1}\rho(z)\left(\frac{|\Ph'(z)|(1-|z|^2)}{1-|\Ph(z)|^2}+\frac{|\psi'(z)(1-|z|^2)}{1-|\psi(z)|^2}\right)=0.$$
\end{theorem}

Before giving the proof, notice that the quantity $\rho(z)$ is simply the pseudo-hyperbolic distance between $\Ph(z)$ and $\psi(z)$ and thus we have the well known equality $$1-\rho^2(z)=\frac{(1-|\Ph(z)|^2)(1-|\psi(z)|^2)}{|1-\overline{\Ph(z)}\psi(z)|^2}.$$

\begin{proof}
We again argue by contrapositive and assume that \begin{equation}\label{eqn:Dpseudo}\lim_{|z|\rightarrow1}\rho(z)\left(\frac{|\Ph'(z)|(1-|z|^2)}{1-|\Ph(z)|^2}+\frac{|\psi'(z)(1-|z|^2)}{1-|\psi(z)|^2}\right)\neq 0.\end{equation} To show that $\CO-\CP$ is not compact, we will show that there is a sequence $\{z_n\}$ in $\D$ with $|z_n|\rightarrow1$ such that $$\left\|(\CO-\CP)^*\left(\frac{K_{z_n}^{(1)}}{\|K_{z_n}^{(1)}\|}\right)\right\|\not\rightarrow0.$$ Since $K_w^{(1)}/\|K_w^{(1)}\|\rightarrow 0$ weakly as $|w|\rightarrow 1$, we will conclude that $(\CO-\CP)^*$, and hence $\CO-\CP$, is not compact.

As in the proof of Lemma \ref{lem:Dfod}, for $w\in\D$ we have $$\left\|(\CO-\CP)^*\left(\frac{K_w^{(1)}}{\|K_w^{(1)}\|}\right)\right\|^2=\frac{\|\CO^*K_w^{(1)}-\CP^*K_w^{(1)}\|^2}{\|K_w^{(1)}\|^2}$$ which is greater than or equal to
\begin{equation}\label{eqn:Dpseudo1}\frac{\|\CO^*K_w^{(1)}\|^2}{\|K_w^{(1)}\|^2}+ \frac{\|\CP^*K_w^{(1)}\|^2}{\|K_w^{(1)}\|^2}-2\left|\left\langle\frac{\CO^*K_w^{(1)}}{\|K_w^{(1)}\|},\frac{\CP^*K_w^{(1)}}{\|K_w^{(1)}\|}\right\rangle\right|.\end{equation} Manipulating the third term here, $$\left|\left\langle\frac{\CO^*K_w^{(1)}}{\|K_w^{(1)}\|},\frac{\CP^*K_w^{(1)}}{\|K_w^{(1)}\|}\right\rangle\right|=\frac{|\Ph'(w)||\psi'(w)|(1-|w|^2)^2}
{|1-\overline{\Ph(w)}\psi(w)|^2}$$ which is equal to
$$\frac{(1-|\Ph(w)|^2)(1-|\psi(w)|^2)}{|1-\overline{\Ph(w)}\psi(w)|^2}\frac{|\Ph'(w)|(1-|w|^2)}{1-|\Ph(w)|^2}\frac{|\psi'(w)|(1-|w|^2)}{1-|\psi(w)|^2}$$ or, more simply, $$(1-\rho^2(w))\frac{\|\CO^*K_w^{(1)}\|}{\|K_w^{(1)}\|}\frac{\|\CP^*K_w^{(1)}\|}{\|K_w^{(1)}\|}.$$ Substituting into the expression in ~\eqref{eqn:Dpseudo1} and factoring we see that $$\frac{\|\CO^*K_w^{(1)}-\CP^*K_w^{(1)}\|^2}{\|K_w^{(1)}\|^2}\geq\left(\frac{\|\CO^*K_w^{(1)}\|-\|\CP^*K_w^{(1)}\|}{\|K_w^{(1)}\|}\right)^2+
2\rho^2(w)\frac{\|\CO^*K_w^{(1)}\|}{\|K_w^{(1)}\|}\frac{\|\CP^*K_w^{(1)}\|}{\|K_w^{(1)}\|}.$$

As we are assuming that the limit in Eqn. ~\eqref{eqn:Dpseudo} is not 0, it must be the case that there is a sequence $\{z_n\}$ in $\D$ with $|z_n|\rightarrow 1$ such that either $$\rho(z_n)\left(\frac{|\Ph'(z_n)|(1-|z_n|^2)}{1-|\Ph(z_n)|^2}\right)=\rho(z_n)\frac{\|\CO^*K_{z_n}^{(1)}\|}{\|K_{z_n}^{(1)}\|}:=a_n$$
or $$\rho(z_n)\left(\frac{|\psi'(z_n)|(1-|z_n|^2)}{1-|\psi(z_n)|^2}\right)=\rho(z_n)\frac{\|\CP^*K_{z_n}^{(1)}\|}{\|K_{z_n}^{(1)}\|}:=b_n$$ does not converge to 0. Since both the sequences $\{a_n\}$ and $\{b_n\}$ are bounded (by the boundedness of $\CO$ and $\CP$ together with the fact that $\rho(z)\leq 1$ for all $z\in\D$), we may, by passing to a subsequence if necessary, assume that $a_n\rightarrow \A$ and $b_n\rightarrow \beta$ with either $\A\neq 0$ or $\beta\neq 0$. By symmetry we may assume that $\A\neq 0$. By passing to a further subsequence if necessary, we may also assume that $\rho(z_n)\rightarrow p$.

First we note that $p>0$. Indeed, if $p=0$, then it must be the case that $$\lim_{n\rightarrow\infty}\frac{|\Ph'(z_n)|(1-|z_n|^2)}{1-|\Ph(z_n)|^2}=\lim_{n\rightarrow\infty}\frac{\|\CO^*K_{z_n}^{(1)}\|}{\|K_{z_n}^{(1)}\|}=\infty,$$ but this contradicts the fact that $\CO$ is bounded.  Thus we may (by passing to another subsequence if necessary) assume that $$\lim_{n\rightarrow\infty}\frac{|\Ph'(z_n)|(1-|z_n|^2)}{1-|\Ph(z_n)|^2}=\lim_{n\rightarrow\infty}\frac{\|\CO^*K_{z_n}^{(1)}\|}{\|K_{z_n}^{(1)}\|}=\frac{\A}{p}.$$  Similarly, the boundedness of $\CP$ implies that we may assume $$\lim_{n\rightarrow\infty}\frac{|\psi'(z_n)|(1-|z_n|^2)}{1-|\psi(z_n)|^2}=\lim_{n\rightarrow\infty}\frac{\|\CP^*K_{z_n}^{(1)}\|}{\|K_{z_n}^{(1)}\|}=\frac{\beta}{p}.$$

Now, if $\A\neq \beta$, then $$\lim_{n\rightarrow\infty}\left(\frac{\|\CO^*K_{z_n}^{(1)}\|-\|\CP^*K_{z_n}^{(1)}\|}{\|K_{z_n}^{(1)}\|}\right)^2=\left(\frac{\A}p-\frac{\beta}p\right)^2\neq 0.$$ On the other hand, if $\A=\beta\neq 0$, then $$\lim_{n\rightarrow\infty}2\rho^2(z_n)\frac{\|\CO^*K_{z_n}^{(1)}\|}{\|K_{z_n}^{(1)}\|}\frac{\|\CP^*K_{z_n}^{(1)}\|}{\|K_{z_n}^{(1)}\|}=2\A\beta\neq 0.$$ In either case, $$\left\|(\CO-\CP)^*\left(\frac{K_{z_n}^{(1)}}{\|K_{z_n}^{(1)}\|}\right)\right\|\not\rightarrow0$$ as desired.\end{proof}

In the case of $S^2$, the proofs are nearly identical to those just given for $\DI$ except for the fact that the kernel functions take a slightly different form in $S^2$; for a reference we point the reader to \cite{KH}.  Notice also that the hypotheses of Lemma \ref{lem:Sfod} are slightly altered.  This is due to the fact that on $S^2$, the boundedness of $\CO$ implies that $\Ph$ has finite angular derivative at any point $\zeta\in\partial\D$ with $|\Ph(\zeta)|=1$ (\cite{CMC} Theorem 4.13).

\begin{lemma}\label{lem:Sfod}
Let $\Ph$ and $\psi$ be analytic self-maps of $\D$ such that $\CO$ and $\CP$ are bounded on $S^2$. Further assume that there exists $\zeta\in\partial\D$ such that $|\Ph(\zeta)|= |\psi(\zeta)|=1$.  If $\CO-\CP$ is compact on $S^2$, then $\Ph$ and $\psi$ have the same first order data at $\zeta$.
\end{lemma}

The following theorem should be compared to the result of Theorem \ref{thm:Dpseudo} for the Dirichlet space and \cite{JMoore} Theorem 4. Here we see the square of the modulus of the derivative appearing.  Again, this difference in form is due to the use of the kernel for evaluation of the first derivative in $S^2$.

\begin{theorem}\label{thm:Spsuedo}
Let $\Ph$ and $\psi$ be analytic self-maps of $\D$ such that $\CO$ and $\CP$ are bounded on $S^2$.  If $\CO-\CP$ is compact on $S^2$, then $$\lim_{|z|\rightarrow1}\rho(z)\left(\frac{|\Ph'(z)|^2(1-|z|^2)}{1-|\Ph(z)|^2}+\frac{|\psi'(z)|^2(1-|z|^2)}{1-|\psi(z)|^2}\right)=0.$$
\end{theorem}

We close with our main theorem and an interesting corollary for one particular class of maps.

\begin{theorem}\label{thm:DifferenceSOD}Let $\Ph$ and $\psi$ be analytic self-maps of $\D$. Let $F(\Ph)$ be the set of points $\zeta\in\partial\D$ with $|\Ph(\zeta)|=1$ and similarly for $F(\psi)$. Further suppose that $F(\Ph)=F(\psi):=F$.
\begin{enumerate}
\item[(a)] Suppose $\CO$ and $\CP$ are bounded on $S^2=\DI_{-1}$.  If $\Ph$ and $\psi$ have second order data at each point $\zeta\in F$ and $\CO-\CP$ is compact on $S^2$, then $\Ph$ and $\psi$ have the same second order data at each point $\zeta \in F$.

\item[(b)] Let $\gamma >-1$ and suppose $\CO$ and $\CP$ are bounded on $\DA$ for some $\A$ with $-1\leq\A<\gamma$. If $\Ph$ and $\psi$ have finite angular derivative and second order data at each point $\zeta\in F$ and $\CO-\CP$ is compact on $\DI_{\gamma}$, then $\Ph$ and $\psi$ have the same second order data at each point $\zeta \in F$.

\end{enumerate}
\end{theorem}

\begin{proof}
For (a), we know that $\Ph$ and $\psi$ have the same first order data at each $\zeta\in F$ by Lemma \ref{lem:Sfod}.  If $\Ph''(\zeta)\neq\psi''(\zeta)$, then by \cite{JMoore}  Proposition 1 there is a sequence $\{z_n\}$ in $\D$ with $z_n\rightarrow \zeta$, i.e. $|z_n|\rightarrow 1$,such that $\rho(z_n)\not\rightarrow 0$ and $$\frac{1-|z_n|^2}{1-|\Ph(z_n)|^2}\not\rightarrow 0.$$  Since $\Ph'(\zeta)\neq0$, this contradicts Theorem \ref{thm:Spsuedo} as we are assuming that $\CO-\CP$ is compact on $S^2$, and hence $\Ph''(\zeta)=\psi''(\zeta).$

For (b), it suffices to prove the result for $\gamma$ with $-1<\gamma\leq 1$; the other cases follow from  \cite{JMoore} Theorem 6 (we will actually appeal to this result to obtain our conclusion momentarily).  Choose $\beta$ and $\delta$ such that $-1\leq\A<\gamma\leq1<\delta<\beta$.  Then it must be the case that $\DI_{\gamma}$ and $\DI_{\delta}$ are interpolation spaces for the pair $[\DA,\DI_{\beta}]$.  Furthermore, by Theorem \ref{thm:Cwikelcompact}, we know that $\CO-\CP$ is compact on $\DI_{\delta}$. Appealing to \cite{CMC} Theorem 9.16 we see that $\Ph$ and $\psi$ must have the same first order data at each $\zeta\in F$; \cite{JMoore} Theorem 6 guarantees us that $\Ph$ and $\psi$ have the same second order data at each $\zeta\in F$.
\end{proof}

As in \cite{JMoore}, we can extend the previous result to a stronger result for the class of linear fractional symbols. We first have one final lemma which is also a known result on a variety of spaces; we include references where appropriate and a proof for the spaces not explicitly referenced.  Recall again that any nonconstant linear fractional self-map of $\D$ is univalent and induces a bounded composition operator on each weighted Dirichlet space.

\begin{lemma}\label{lem:linearfractional2}
Let $\gamma\geq -1$ and let $\Ph$ be a linear fractional self-map of $\D$. Then $\CO$ is compact on $\DI_{\gamma}$ if and only if $\|\Ph\|_{\infty}<1$.
\end{lemma}

\begin{proof}
If $-1\leq\gamma<0$, then the statement follows from \cite{JS1} Theorem 2.1 and the remarks following the theorem.  For $\gamma=1$, the Hardy space case, see \cite{JSComp} (pages 23, 29 - 31); the case for $\gamma>1$ follows similarly.

Finally, when $0\leq \gamma<1$, choose $\beta>1$. The spaces $\DI_{\gamma}$ and $\DI_1=H^2$ are interpolation spaces for the pair $[S^2,\DI_{\beta}]$; recall $S^2=\DI_{-1}$.  Now, if $\CO$ is compact on $\DI_{\gamma}$, it follows from Theorem \ref{thm:Cwikelcompact} that $\CO$ is compact on $H^2$. Then $\|\Ph\|_{\infty}<1$ since $\Ph$ is linear fractional. Conversely, suppose that $\|\Ph\|_{\infty}<1$.  Then $\CO$ is compact on $H^2$ and hence on $\DI_{\gamma}$ by Theorem \ref{thm:Cwikelcompact}.
\end{proof}

\begin{cor}\label{cor:linearfractional}
Let $\Ph$ and $\psi$ be linear fractional self-maps of $\D$ and let $\gamma\geq -1$.  Then $\CO-\CP$ is compact on $\DI_{\gamma}$  if and only $\Ph=\psi$ or both $\CO$ and $\CP$ are compact on $\DI_{\gamma}$.
\end{cor}

\begin{proof}
If $\Ph=\psi$ or both $\CO$ and $\CP$ are compact on $\DI_\gamma$, then it is clear that $\CO-\CP$ is compact on $\DI_{\gamma}$.  Conversely, suppose that $\CO-\CP$ is compact on $\DI_{\gamma}$.  If either $\CO$ or $\CP$ is compact on $\DI_{\gamma}$, then both $\CO$ and $\CP$ are compact on $\DI_{\gamma}$ by the fact that the compact operators form a linear subspace within the set of all bounded operators on $\DI_{\gamma}$.  Thus we may further assume that neither $\CO$ nor $\CP$ is compact on $\DI_{\gamma}$.  It follows now by Lemma \ref{lem:linearfractional2} that $\|\Ph\|_{\infty}=1$ and $\|\psi\|_{\infty}=1$, i.e. $\Ph$ and $\psi$ have contact with the boundary of $\D$.

We now consider a few cases. For $\gamma>1$, see \cite{JMoore} Corollary 2.   Next assume that $\gamma=-1$.  Then Lemma \ref{lem:Sfod} implies that $\Ph$ and $\psi$ have the same first order data.  Applying Theorem \ref{thm:DifferenceSOD}(a) we see that $\Ph$ and $\psi$ must also have the same second order data and hence $\Ph=\psi$ by Lemma \ref{lem:linearfractional1}. A similar argument holds for $\gamma=0$ but we will use interpolation.

Next consider $-1<\gamma\leq 1$.  We again use an interpolation scheme and choose $\delta$ and $\beta$ with $-1<\gamma\leq1<\delta<\beta$.  Here we have $\DI_{\gamma}$ and $\DI_{\delta}$ as interpolation spaces for the pair $[S^2,\DI_{\beta}]$.  Since $\CO-\CP$ is compact on $\DI_{\gamma}$, it follows by Theorem \ref{thm:Cwikelcompact} that $\CO-\CP$ is compact on $\DI_{\delta}$ and thus $\Ph$ and $\psi$ have the same first order data by \cite{CMC} Theorem 9.16.  Theorem \ref{thm:DifferenceSOD}(b) then implies that $\Ph$ and $\psi$ have the same second order data and hence $\Ph=\psi$ by Lemma \ref{lem:linearfractional1}.
\end{proof}

\section*{Acknowledgments}
Part of this work is taken from the second author's Ph.D.~dissertation written at the University of Virginia under the direction of Professor Barbara D. MacCluer.

\bibliographystyle{amsplain}
\bibliography{references.bib}
\end{document}